\documentclass[11pt]{amsart}
\usepackage{mathrsfs}
\usepackage{amssymb}
\usepackage[T1]{fontenc}
\usepackage[all]{xy}
\usepackage{amscd}
\usepackage{amsmath, amssymb}
\usepackage{amsthm}
\usepackage{amsfonts}
\usepackage[colorlinks,linkcolor=blue,citecolor=blue, pdfstartview=FitH]
{hyperref}
\usepackage{backref}
\usepackage{amscd}
\usepackage{easybmat}
\usepackage{mathrsfs}
\usepackage{amsfonts}
\usepackage{color}
\usepackage{pifont}
\usepackage{upgreek}
\usepackage{bm}
\usepackage{hyperref}
\usepackage{shorttoc}
\usepackage{amsmath,amstext,amsthm,a4,amssymb,amscd}
\usepackage[mathscr]{eucal}
\usepackage{mathrsfs}
\usepackage{epsf}

\numberwithin{equation}{section}

\def\p{\partial}
\def\o{\overline}
\def\b{\bar}
\def\mb{\mathbb}
\def\mc{\mathcal}

\theoremstyle{plain}
\newtheorem{thm}{Theorem}[section]
\newtheorem{lemma}[thm]{Lemma}
\newtheorem{prop}[thm]{Proposition}

\theoremstyle{definition}

\theoremstyle{definition}
\newtheorem{defn}[thm]{Definition}
\newcommand{\comment}[1]{}

\def\op{\operatorname}
\usepackage{fancyhdr}
\newenvironment{aligns}{\equation\aligned}{\endaligned\endequation}

\makeatletter
\@namedef{subjclassname@2020}{2020 Mathematics Subject Classification}

\begin{document}

\title{Harmonic maps between  surfaces homotopic to a  (branched) covering map}
\author{InKang Kim}
\author{Xueyuan Wan}

\address{Inkang Kim: School of Mathematics, KIAS, Heogiro 85, Dongdaemun-gu Seoul, 02455, Republic of Korea}
\email{inkang@kias.re.kr}

\address{Xueyuan Wan: School of Mathematics, Korea Institute for Advanced Study, Seoul 02455, Republic of Korea}
\email{xwan@kias.re.kr}

\begin{abstract}

In the paper, we consider the harmonic maps 
 between surfaces $\Sigma$ and $S$ in the homotopy class of a (branched) covering map $u_0$. We prove the uniqueness of critical points of energy function and the injectivity of Hopf differential if $u_0$ is a covering map.  On the other hand, if $u_0$ is a branched covering, we show that the uniqueness of critical points fails if $u_0$ is a non-simple branched covering, and prove the injectivity of Hopf differential $\Phi:\mc{T}(S)\to \op{QD}(\Sigma,g)$ when $g=[u_0^* h]$ for some hyperbolic metric $h$ on $S$.
 
\end{abstract}

  \thanks{Research by Inkang Kim is partially supported by Grant NRF-2019R1A2C1083865 and KIAS Individual Grant (MG031408), and Xueyuan Wan is supported by KIAS Individual Grant (MG075801) at Korea Institute for Advanced Study.}

\maketitle

\section*{Introduction}

Let $u_0:\Sigma\to S$ be a smooth map between two Riemann surfaces of genus $\geq 2$ and $h$ be a hyperbolic metric on $S$. It is known that if $u_0$ is a map of nonzero degree,  it can be homotoped to a map which factors through a pinch map and a branched covering map \cite{Ed}. In this reason we deal with (branched) covering maps.

 Eells and Sampson \cite{ES} proved the existence of  a harmonic map between $\Sigma$ and $(S,h)$ in the   homotopy class $[u_0]$, and Hartman \cite{Hart} proved the uniqueness. Especially when  $u_0$ is the identity map, Schoen-Yau \cite{SY0} and Sampson \cite{Sampson} showed the harmonic map is a diffeomorphism. In this case, Tromba \cite{Tromba} proved the uniqueness of critical points of Dirichlet's energy, Sampson \cite{Sampson} proved the injectivity of the Hopf differential of the harmonic map, and  the surjectivity was proved by Wolf \cite{Wolf0}. In this paper, we will  consider the related problems in the case $u_0$ is a (branched) covering map. 

Firstly, we assume $u_0:\Sigma\to S$ is a covering map. For each $t$ in Teichm\"uller space $\mc{T}(\Sigma)$ of $\Sigma$, there exists a unique harmonic map $u_t$ in the homotopy class $[u_0]$, and the energy $E(u_t)$ depends on $t$ smoothly, hence one obtains a well-defined energy function $E(t)$ on Teichm\"uller space $\mc{T}(\Sigma)$, see e.g. \cite{Tromba}. In \cite{KWZ}, G. Zhang and the authors calculated the second variation of the energy function, and proved the convexity of the energy functions at its critical points.  Furthermore it is strictly convex  when $u_0$ is a covering map. Using the argument of Morse theory, we proved the following result in {\cite[Corollary 0.2]{KWZ}}.
\begin{thm}\label{thm1}
	If $u_0:\Sigma\to S$ is a covering map, then the associated energy function has only one critical point.
\end{thm}
In this paper we will give a direct proof of the above result by using the uniqueness of harmonic maps in a fixed homotopy class \cite{Hart}. More precisely, if $u_1$ and $u_2$ are two harmonic maps homotopic to $u_0$, then $u_i,i=1,2$ are covering maps and $u_i: (\Sigma, u_i^*h)\to (S,h)$ are local isometries. Let $f:(\Sigma,u_1^*h)\to (\Sigma,u_2^*h)$ be the unique harmonic map and homotopic to identity, thus $u_2\circ f:(\Sigma,u_1^*h)\to (S,h)$ is also a harmonic map and homotopic to $u_0$, which implies that $u_1=u_2\circ f$ by the uniqueness of harmonic maps. So $f:(\Sigma,u_1^*h)\to (\Sigma,u_2^*h)$ is a conformal map, and $[u_1^*h]=[u_2^*h]$ in Teichm\"uller space $\mc{T}(\Sigma)$. See Theorem \ref{unique}.
Here  $[\bullet]$ denotes the  class in Teichm\"uller space  $\mc{T}(\Sigma)=\mc{M}(\Sigma)/\op{Diff}_0(\Sigma)$, where $\mc{M}(\Sigma)$ is the space the Riemannian metrics   on $\Sigma$ and $\op{Diff}_0(\Sigma)$ is the group of all  orientation-preserving diffeomorphisms of $\Sigma$ which are homotopic to the identity map where the quotient is given by the equivalence relation that two metrics are conformal under the orientation-preserving diffeomorphisms homotopic to the identity, see e.g. \cite[Theorem 1.8]{IT}. 

If $u_0:\Sigma\to S$ is a branched covering, one may wonder whether the critical point of energy function is unique. On this problem, we prove that the uniqueness fails in general. More precisely,
\begin{prop}\label{prop2}
If $u_0:\Sigma\to S$ is a non-simple branched covering, then the associated energy function has at least two critical points.
\end{prop}
Here, a {\it simple branched covering} is a branched covering that each fiber contains at most one singular point and that of  local degree $2$, see Definition \ref{sbc}, and a {\it non-simple branched covering} is a branched covering which is not simple. 

We also consider the injectivity of Hopf differential. Let $u_0:\Sigma\to S$ be a smooth map with a fixed  hyperbolic metric $g$ on $\Sigma$.  For each hyperbolic metric $h$ on $S$, there exists a unique harmonic map $u_h:(\Sigma,g)\to (S,h)$ which is homotopic to $u_0$. The {\it Hopf differential} of $u_h$ is given by $(u_h^*h)^{2,0}$, which is holomorphic since $u_h$ is a harmonic map. Thus one obtains the following map 
$$\Phi:\mc{T}(S)\to \op{QD}(\Sigma,g), \quad \Phi(h)=(u^*_hh)^{2,0},$$ 
where $\operatorname{QD}(\Sigma,g)$ denotes the space of holomorphic quadratic differentials on $(\Sigma,g)$. Here we use $h$ to denote the point $[h]$ in $\mc{T}(S)$ since each conformal class contains a unique hyperbolic metric. If $u_0$ is the identity map, Sampson \cite{Sampson} showed that $\Phi$ is injective, and we generalize the result to the case $u_0$ is a covering map. 
\begin{thm}\label{thm3}
If $u_0:\Sigma\to S$ is a covering map, then $\Phi$ is injective. 	
\end{thm}
Following \cite{Sampson}, see also \cite{Wolf0}, if $\Phi(h_1)=\Phi(h_2)$ then $u_{h_2}^*h_2=u_{h_1}^*h_1$. Let $f_{h_i}:(\Sigma,g)\to (\Sigma,u_0^*h_i),\, i=1,2$ denote the harmonic maps homotopic to identity, then $u_{h_i}=u_0\circ f_{h_i}$. The key step is that we can show that the map $f_{h_2}\circ f_{h_1}^{-1}$ descends to a map between $(S,h_1)$ and $(S,h_2)$, i.e. the map $F:=u_0\circ (f_{h_2}\circ f_{h_1}^{-1})\circ u_0^{-1}$ is well-defined. It is a conformal map and homotopic to identity, so $h_1=h_2$ in $\mc{T}(S)$. 

It is also natural to ask whether $\Phi$ is injective for  a branched covering $u_0$. Here we can answer this question partially. 
\begin{prop}\label{prop3}
	If $u_0:\Sigma\to S$ is a branched covering, then $\Phi:\mc{T}(S)\to \op{QD}(\Sigma,[u_0^*h_0])$ is injective for any hyperbolic metric $h_0$ on $S$.
\end{prop}

This article is organized as follows. In section \ref{sec1}, we consider the uniqueness of critical points of energy function for covering maps and prove Theorem \ref{thm1}. Then we show that the uniqueness fails for non-simple branched coverings and prove Proposition \ref{prop2}. In section \ref{sec2}, we consider the injectivity of Hopf differential and prove Theorem \ref{thm3}, Proposition \ref{prop3}.

\section{Uniqueness of critical points of energy function}\label{sec1}

Let $u_0:\Sigma\to S$ be a smooth and surjective map. Fix a hyperbolic metric $h=\rho^2|dv|^2$ on $S$, then for any hyperbolic metric $g=\lambda^2|dz|^2$ on $\Sigma$ and any smooth map $u:\Sigma\to S$, the energy of $u$ is given by
\[E(u):=\int_{\Sigma} \rho^{2}(u(z))\left(\left|u_{z}^{v}\right|^{2}+\left|u_{\bar{z}}^{v}\right|^{2}\right) \frac{i}{2} d z \wedge d \bar{z},\]
where $u=(u^v,u^{\b{v}})$ with respect to the local  coordinates $\{v,\b{v}\}$ on $S$, and \(u_{z}^{v}:=\partial_{z} u^{v}\) and \(u_{\bar{z}}^{v}:=\partial_{\bar{z}} u^{v}\). 
Then $u$ is harmonic if it satisfies the following Euler-Lagrange equation
\begin{align}\label{EL}
\p_{\b{z}}u^v_z+\p_v\log\rho^2 u^v_z u^v_{\b{z}}=0.	
\end{align}
For each conformal class on $\Sigma$, there is a unique harmonic map $u_g$ which is homotopic to $u_0$, see \cite{ES, Hart}. Thus, we obtain an energy function 
$E(t), t\in\mc{T}(\Sigma)$ on the Teichm\"uller space of $\Sigma$, see \cite[Page 66-67]{Tromba} and also \cite[Section 1.1]{KWZ}.
For any surjective map $u:\Sigma\to S$, one has
\begin{aligns}\label{min}
E(u) &=\int_{\Sigma}\rho^2(|u^v_z|^2+|u^v_{\b{z}}|^2)\frac{i}{2}dz\wedge d\b{z}\\
&\geq \left|\int_{\Sigma}\rho^2(|u^v_z|^2-|u^v_{\b{z}}|^2)\frac{i}{2}dz\wedge d\b{z}\right|	\quad \text{with equality iff $u$ is $\pm$ holomorphic}\\
&=\left|\int_{\Sigma}u^*(\rho^2\frac{i}{2}dv\wedge d\b{v})\right|=|\deg u|\text{Area}(S).
\end{aligns}
Hence if $u$ is $\pm$holomorphic, then the energy attains the minimum, hence it is a critical point in the Teichm\"uller space.
 Firstly, we give the definition of covering maps, see e.g. \cite[Definition 1.3.2]{Jost}.
\begin{defn}[Covering map]
	 A local homeomorphism \(u: \Sigma \rightarrow S\) is called a covering
if each \(q \in S\) has a (connected) neighbourhood \(V\) such that every connected
component of \(u^{-1}(V)\) is mapped by \(u\) homeomorphically onto \(V\). 
\end{defn}
 Now we give a direct proof of \cite[Corollary 0.2]{KWZ}.
\begin{thm}\label{unique}
	If $u_0:\Sigma\to S$ is a smooth covering map, then the associated energy function has only one critical point.
\end{thm}
\begin{proof}
Since $u_0$ is a covering map, so the pullback metric $u^*_0h$ is a Riemannian metric on $\Sigma$. Denote by $[\bullet]$ denotes the conformal class in Teichm\"uller space $\mc{T}(\Sigma)=\mc{M}(\Sigma)/\op{Diff}_0(\Sigma)$.  Note that $u_0:(\Sigma,u_0^*h)\to (S,h)$ is a local isometry, so $u_0$ is $\pm$-holomorphic between $(\Sigma,[u_0^*h])$ and $(S,[h])$, see e.g. \cite[Section 1.4]{Wood}, which implies that $[u_0^*h]\in \mc{T}(\Sigma)$ is a critical point of the energy function associated to the fixed $u_0$ due to Equation (\ref{min}), see also  \cite{SU,SU2}, and \cite[Proposition 2.2]{KWZ}.

Now we assume that there exists another critical point $t\in\mc{T}(\Sigma)$, and denote by $u_{t}$ the associated harmonic map corresponding to $t$,  homotopic to $u_0$, then  $u_t: (\Sigma,t)\to (S,[h])$ is $\pm$-holomorphic. Since $u_0$ is a covering map, so is $u_t$. Thus $u_t^*h$ is a hyperbolic metric on $\Sigma$, and using $u_t=u_t\circ \op{Id}$ and the fact that both $u_t:(\Sigma,t)\to (S, [h])$ and $u_t:(\Sigma, [u_t^*h])\to (S,[h])$ are $\pm$-holomorphic,
\[
\xymatrix{
(\Sigma,t)\ar[rr]^{\op{Id}}\ar[rd]_{u_t}&&(\Sigma,[u_t^*h])\ar[ld]^{u_t}\\
&(S,[h])
}
\]
we get that the identity map \text{Id} between the Riemann surfaces $(\Sigma,t)$ and $(\Sigma, [u_t^*h])$ is biholomorphic, 
 and so $t=[u^*_th]$. Consider the following diagram:
 \[
\xymatrix{
(\Sigma,u_0^*h)\ar[rr]^{f}\ar[rd]_{u_0}&&(\Sigma,u_t^*h)\ar[ld]^{u_t}\\
&(S,h)
}
\]
Denote by $f$ the unique harmonic map between $(\Sigma,u_0^*h)$ and $(\Sigma,u_t^*h)$, and is homotopic to identity. Since $u_t$ is a local isometry,  $u_t\circ f$ is a harmonic map and homotopic to $u_0$, which follows that $u_0=u_t\circ f$ by the uniqueness of harmonic map \cite{Hart}. So $u_0^*h=f^*u_t^*h$, which means that $f$ is a conformal map and homotopic to identity. Thus $[u_0^*h]=[u_t^*h]=t\in\mc{T}(\Sigma)$, which completes the proof. 
\end{proof}

Next we consider a fixed branched covering map $u_0$. One can refer \cite{BE, BM} for the details on branched coverings. 

\begin{defn}[Branched covering]
A smooth and surjective map $u: \Sigma \to S$ is called a {\it branched covering} if there exists a finite set, called {\it branch set},  $B\subset S$, such that $u: \Sigma \backslash u^{-1}(B) \to S\backslash B$ is a covering map, and for each point $p$ in $u^{-1}(B)$, called a {\it singular point}, the map $u$ about $p$ is locally equivalent to the map $z\mapsto z^k (z\in\mb{C})$ about $0\in\mb{C}$ for some positive integer $k\geq 2$. Here $k$ is called a local degree, denoted by
deg($u, p$).
\end{defn}
In \cite{BE}, one can find the following definitions of $b$-homotopic and simple branched covering.  
\begin{defn}[$b$-homotopic]
	Two branched coverings \(u_{0}, u_{1}: \Sigma \rightarrow S\) 
are {\it \(b\)-homotopic} if there is a homotopy \(\theta_{t}: \Sigma \rightarrow S, 0 \leqslant t \leqslant 1,\) such that
\(\theta_{0}=u_{0}, \,\,\theta_{1}=u_{1},\) and each \(\theta_{t}\) is a branched covering.
\end{defn}
\begin{defn}[Simple branched covering]\label{sbc}
	A branched covering \(u: \Sigma \rightarrow S\) of absolute degree (maximum cardinality of a fiber) \(n \geqslant 2\) is {\it simple} provided that for each \(y \in S\) the fiber
\(u^{-1}(y)\) over \(y\) consists of at least \(n-1\) points (and hence contains at most
one singular point  of local degree $2$).
\end{defn}
It turns out that the simple branched coverings are generic among branched coverings. That is, they form an open and dense set in the space of all branched coverings between two given Riemann surfaces. 
\begin{prop}[{\cite[Proposition 3]{BE}}]\label{prop-b}
Any branched covering \(u: \Sigma \rightarrow S\) is $b$-homotopic to a
simple branched covering by an arbitrarily small homotopy.	
\end{prop}
In general, if $u_0$ is a branched covering, then the critical points of the associated energy function need not be unique. More precisely,
\begin{prop}
If $u_0$ is a smooth non-simple branched covering, then the associated energy function has at least two critical points. 
\end{prop}
\begin{proof}
Since $u_0$ is a branched covering, then $u_0:(\Sigma,[u_0^*h])\to(S,h)$ is $\pm$-holomorphic, and $[u_0^*h]\in\mc{T}(\Sigma)$ is a critical point of the associated energy function $E(t)$, $t\in\mc{T}(\Sigma)$ due to Equation (\ref{min}). 
From Proposition \ref{prop-b}, $u_0$ is $b$-homotopic to a simple branched covering $u:\Sigma\to S$. So $[u^*h]\in\mc{T}(\Sigma)$ is also a critical point of the energy function. 

We claim $[u^*h]\neq [u_0^*h]$. In fact, if $[u^*h]= [u_0^*h]$, then there exists a biholomorphic mapping $f:(\Sigma,[u_0^*h])\to (\Sigma,[u^*h])$ and is homotopic to identity,
 \[
\xymatrix{
(\Sigma,[u_0^*h])\ar[rr]^{f}\ar[rd]_{u_0}&&(\Sigma,[u^*h])\ar[ld]^{u}\\
&(S,h)
}
\]
This implies that $u\circ f$ is $\pm$-holomorphic and homotopic to $u_0$. By uniqueness of harmonic map \cite{Hart}, $u_0=u\circ f$. On the other hand, $u_0$ is a non-simple branched covering while $u$ is simple, so there exists a singular point $p$ of $u_0$ such that the local degree satisfies $\deg(u_0,p)> 2$. Then
\begin{equation*}
  2< \deg(u_0,p)=\deg(u,f(p))\cdot\deg(f,p)=\deg(u,f(p))\leq 2,
\end{equation*}
where the first equality follows from \cite[Lemma A.17]{BM}, a
contradiction. Thus $[u^*h]\neq [u_0^*h]$ and the proof is complete.
\end{proof}

\section{Injectivity of Hopf differential}\label{sec2}
Let $u_0:\Sigma\to S$ be a smooth and surjective map with a fixed  hyperbolic metric $g=\lambda^2|dz|^2$ on $\Sigma$. For any hyperbolic metric $h$ on $S$, there exists a unique harmonic map 
$$u_h:(\Sigma,g)\to (S,h=\rho^2|dv|^2)$$
which is homotopic to $u_0$, and  the {\it Hopf differential} $(u_h^*h)^{2,0}$ is holomorphic. Thus, one obtains the following map:  
\begin{align}
\Phi:	\mc{T}(S)\to \operatorname{QD}(\Sigma,g),\quad \Phi(h)=(u_h^*h)^{2,0},
\end{align}
where $\operatorname{QD}(\Sigma,g)$ denotes the space of holomorphic quadratic differentials on $(\Sigma,g)$. 

Now we assume $u_0$ is a covering map, then there exists a unique harmonic map $$f_h: (\Sigma, g)\to (\Sigma, u_0^*h),$$
which is homotopic to identity and satisfies $u_h=u_0\circ f_h$. For the harmonic map $u_h:\Sigma\to S$, there are two natural smooth functions on $\Sigma$ associated $u_h$,
$$H(z)=\frac{\rho^2(u_h(z))}{\lambda^2(z)}|(u_h)^v_z|^2=\frac{(\rho^2\circ u_0|(u_0)^v_w|^2)( f_h(z))}{\lambda^2(z)}|(f_h)^w_{z}|^2$$
and 
$$L(z)=\frac{\rho^2(u_h(z))}{\lambda^2(z)}|(u_h)^v_{\b{z}}|^2=\frac{(\rho^2\circ u_0 |(u_0)^v_w|^2)( f_h(z))}{\lambda^2(z)}|(f_h)^w_{\b{z}}|^2,$$
which are also  smooth functions associated to $f_h$. Here, without loss of generality, we assume that $u_0:(\Sigma, u_0^*h)\to (S,h)$ is holomorphic.
 The Jacobian, energy density and pullback metric of $u_h$ (or $f_h$) are given by
\begin{align}\label{3.1}
J_h:=H-L,\quad e_h=H+L,\quad f_h^*u_0^*h=\Phi(h)+\o{\Phi(h)}+\lambda^2 e_h.
\end{align}
Moreover, one has 
\begin{align*}\Delta \log H=2 H-2 L-2,\quad \text{where}\, H>0,\end{align*}
where $\Delta:=\frac{4}{\lambda^2}\frac{\p^2}{\p z\p\b{z}}$, and 
\begin{align}\label{3.2}|\Phi(h)|^2=\frac{\rho^4}{\lambda^4}|(u_h)^v_z(u_h)^v_{\b{z}}|^2=HL,\end{align}
see  \cite{Sampson, SY0, Wolf0} for related identities. 
Since $f_h$ is a harmonic map and is homotopic to identity, so $f_h$ is a  diffeomorphism  and $H> 0$ on  $\Sigma$, see \cite{Sampson, SY0}.
Let $f_{h_1}, f_{h_2}$ be two harmonic maps associated to two hyperbolic metrics $h_1,h_2$ on $S$, respectively. 
\begin{lemma}\label{lemma1}
If $\Phi(h_1)=\Phi(h_2)$, then $(f_{h_2}\circ f_{h_1}^{-1})^*u_0^*h_2=u_0^*h_1$.
\end{lemma}
\begin{proof}
	We will follow the proof of \cite[Theorem 12]{Sampson}, see also the proof of \cite[Theorem 3.1]{Wolf0}.  Denote by $H_i$, $i=1,2$ the two smooth functions associated to the harmonic maps $u_{h_i}=u_0\circ f_{h_i}$ ,$i=1,2$. If $\Phi(h_1)=\Phi(h_2)=\Phi$, we  claim that $H_1=H_2$.  If $H_1>H_2$ at some point in $\Sigma$, then at the maximum point $p$ of $\frac{H_1}{H_2}$, one has
	$$0\geq (\Delta\log\frac{H_1}{H_2})(p)=2(H_1-H_2)(p)-2|\Phi|^2(\frac{1}{H_1}-\frac{1}{H_2})(p)>0,$$	
	a contradiction, so $H_1\leq H_2$. Similarly, one has $H_2\leq H_1$. Thus $H_1=H_2$. From \eqref{3.1} and \eqref{3.2}, we obtain $f_{h_1}^*u_0^*h_1=f_{h_2}^*u_0^*h_2$, which completes the proof.
\end{proof}

Now we consider the following diagram
\begin{equation*}
\begin{CD}
(\Sigma,g) @>f_{h_1}>> (\Sigma,u^*_0h_1)@>u_0>>(S,h_1)\\
@V\op{Id}VV @Vf_{h_2}\circ f_{h_1}^{-1} VV @V F_1 VV\\
(\Sigma, g) @>f_{h_2}>> (\Sigma,u^*_0h_2)@>u_0>> (S,h_2)
\end{CD} 
\end{equation*}
\begin{lemma}\label{prop11}
If $f_{h_2}\circ f_{h_1}^{-1}$ is an isometry and homotopic to identity, then
$f_{h_2}\circ f_{h_1}^{-1}$ preserves the fibers of $u_0$, i.e. $u_0\circ f_{h_2}\circ f_{h_1}^{-1}(p)=u_0\circ f_{h_2}\circ f_{h_1}^{-1}(q)$  if $u_0(p)=u_0(q)$. Thus $$F=u_0\circ (f_{h_2}\circ f_{h_1}^{-1})\circ u_0^{-1}:(S,h_1)\to (S,h_2)$$	
is well-defined and homotopic to identity.
\end{lemma}
\begin{proof}
Let $F_1:(S,h_1)\to (S,h_2)$ be the unique harmonic map and  homotopic to the identity $\op{Id}$.  Since $F_1\circ u_0:(\Sigma,u^*_0h_1)\to (S,h_2)$ has a homotopy 
$$F(t)\circ u_0:(\Sigma,u^*_0h_1)\to (S,h_2),\quad F(0)=\op{Id},\quad F(1)=F_1,\quad  t\in [0,1]$$
and $u_0=F(0)\circ u_0$ has a canonical lifting $\op{Id}:(\Sigma, u_0^*h_1)\to (\Sigma,u_0^*h_2)$, by Homotopy lifting, there is a unique homotopy lifting $\tilde{F}(t)$ of $F(t)\circ u_0$, which is homotopic to $\op{Id}$, so 
$
u_0\circ\tilde{F}(t)=F(t)\circ u_0.
$
Now take $t=1$, we have 
\begin{equation}\label{3.3}
  u_0\circ\tilde{F}=F_1\circ u_0,
\end{equation}
 where $\tilde{F}:=\tilde{F}(1):(\Sigma, u_0^*h_1)\to (\Sigma,u_0^*h_2)$. Since $u_0$ is a local isometry,  $\tilde{F}$ is also a harmonic map. On the other hand, $f_{h_2}\circ f_{h_1}^{-1}$ is an isometry and homotopic to identity, so $f_{h_2}\circ f_{h_1}^{-1}$ is also a harmonic map between $(\Sigma, u_0^*h_1)$ and $(\Sigma, u_0^*h_2)$ and is homotopic to identity. By the uniqueness of harmonic maps, one has
$\tilde{F}=f_{h_2}\circ f_{h_1}^{-1}$. Combining with \eqref{3.3} we obtain
$u_0\circ f_{h_2}\circ f_{h_1}^{-1}(p)=u_0\circ f_{h_2}\circ f_{h_1}^{-1}(q)$  if $u_0(p)=u_0(q)$. Thus $F=u_0\circ (f_{h_2}\circ f_{h_1}^{-1})\circ u_0^{-1}$
is well-defined, and so $F=F_1$.
\end{proof}

\begin{thm}
If $u_0:\Sigma\to S$ is a covering map, then the associated Hopf differential 	
\begin{align*}
	\Phi: \mc{T}(S)\to \op{QD}(\Sigma,g),\quad h \mapsto (u^*_hh)^{2,0}
\end{align*}
is injective.
\end{thm}
\begin{proof}
From Lemma \ref{lemma1} if $\Phi(h_1)=\Phi(h_2)$, then $f_{h_2}\circ f_{h_1}^{-1}$ is an isometry. By Lemma \ref{prop11}, then 	$f_{h_2}\circ f_{h_1}^{-1}$ preserves the fibers, similar is for $f_{h_1}\circ f_{h_2}^{-1}$. Thus $$F=u_0\circ (f_{h_2}\circ f_{h_1}^{-1})\circ u_0^{-1}:(S,h_1)\to (S,h_2)$$
is well-defined and with inverse $F^{-1}=u_0\circ (f_{h_1}\circ f_{h_2}^{-1})\circ u_0^{-1}$. Moreover, 
$$F^*h_2=(u_0\circ (f_{h_2}\circ f_{h_1}^{-1})\circ u_0^{-1})^*h_2=(u_0^{-1})^*(f_{h_2}\circ f_{h_1}^{-1})^*(u_0)^*h_2=(u_0^{-1})^*u_0^*h_1=h_1.$$
Thus $h_1=h_2\in\mc{T}(S)$. 
\end{proof}

We are wondering if the associated Hopf differential map is injective for branched coverings. For this problem we can prove
\begin{prop}
If $u_0:\Sigma\to S$ is a branched covering, then the associated Hopf differential 
	\begin{align*}
	\Phi: \mc{T}(S)\to \op{QD}(\Sigma,[u_0^*h_0]),\quad h \mapsto (u^*_hh)^{2,0}
\end{align*}
is injective for any hyperbolic metric $h_0$ on $S$.
\end{prop}
\begin{proof}
	If $u_0:\Sigma\to S$ is a branched covering, then $u_0:(\Sigma,[u_0^*h_0])\to (S,h_0)$ is $\pm$-holomorphic for any hyperbolic metric $h_0$ on $S$. Denote by $h_1$ and $h_2$ the two hyperbolic metrics on $S$ such that $\Phi(h_1)=\Phi(h_2)$, and $u_i:(\Sigma,[u_0^*h_0])\to (S,h_i), i=1,2$ are the associated harmonic maps  homotopic to $u_0$. Let $f_i:(S,h_0)\to (S,h_i),i=1,2$ and $f:(S,h_1)\to (S,h_2)$ be the harmonic maps homotopic to identity. Please see the following diagram:
\[
\xymatrix@C=30pt@R=15pt{
&(\Sigma,[u_0^*h_0])\ar[lddd]_{u_1}\ar[rddd]^{u_2}\ar[dd]_{u_0}\\
\\
&(S,h_0)\ar[ld]_{f_1}\ar[rd]^{f_2}\\
(S,h_1)\ar[rr]^f&& (S,h_2)
}
\]
It is well known that the post-composition of harmonic map with a $\pm$-holomorphic map is also harmonic. Thus $f_i\circ u_0, i=1,2$ are harmonic maps, which implies $u_i=f_i\circ u_0$ by the uniqueness of harmonic maps, and $\Phi(h_1)=\Phi(h_2)$ is equivalent to $u_0^*(f_1^*h_1)^{2,0}=u_0^*(f_2^*h_2)^{2,0}$. This implies $(f_1^*h_1)^{2,0}=(f_2^*h_2)^{2,0}$. In fact, for any holomorphic quadratic differentials $\sigma_i(v)dv^2,i=1,2$ on $S$,  
$$u_0^*(\sigma_i(v)dv^2)=\sigma_i(u(z))((u_0)^v_z)^2dz^2.$$
Here, without loss of generality, we assume $u_0$ is holomorphic. If $u_0^*(\sigma_1(v)dv^2)=u_0^*(\sigma_2(v)dv^2)$, then $\sigma_1(v)=\sigma_2(v)$ outside the branch set of $u_0$. Since $\sigma_1(v)$ and $\sigma_2(v)$ are local holomorphic and uniformly bounded functions, so is $\sigma_1(v)-\sigma_2(v)$. By Riemann Removable Singularity Theorem, one has $\sigma_1(v)\equiv \sigma_2(v)$ on $S$. This completes the proof of $(f_1^*h_1)^{2,0}=(f_2^*h_2)^{2,0}$. From \cite[Theorem 12]{Sampson}, one has $h_1=h_2$ in $\mc{T}(S)$.
\end{proof}

\end{document}